\documentclass[11pt,draftcls,onecolumn]{IEEEtran}
\usepackage{graphicx}          
\usepackage[dvips]{epsfig}    
\usepackage{epstopdf}

\newtheorem{remark}{Remark}
\newtheorem{defi}{Definition}
\newtheorem{example}{Example}
\newtheorem{theom}{Theorem}
\newtheorem{prop}{Proposition}

\newtheorem{lem}{Lemma}
\usepackage{url}
\usepackage{amsmath}
\usepackage{amssymb}
\usepackage{color}
\newcommand{\rmjj}[1]{{\color{black} #1}}
\newcommand{\rmj}[1]{{\color{black} #1}}
\newcommand{\mrr}[1]{{\color{black} #1}}
\newcommand{\aaa}[1]{{\color{black} #1}}

\newcommand{\re}{{\mathbb R}}

\newcommand{\n}{{\mathbb N}}

\newcommand{\vectun}{{e }}
\providecommand{\com[2]}{\begin{tt}[#1: #2]\end{tt}}

\newcommand{\fulllanguage}{{\tt full language problem}}
\newcommand{\setmat}{{\Sigma{}}}

\makeatletter
\DeclareRobustCommand{\qed}{%
  \ifmmode 
  \else \leavevmode\unskip\penalty9999 \hbox{}\nobreak\hfill
  \fi
  \quad\hbox{\qedsymbol}}
\newcommand{\openbox}{\leavevmode
  \hbox to.77778em{%
  \hfil\vrule
  \vbox to.675em{\hrule width.6em\vfil\hrule}%
  \vrule\hfil}}
\newcommand{\qedsymbol}{\openbox}
\newenvironment{proof}[1][\proofname]{\par
  \normalfont
  \topsep6\p@\@plus6\p@ \trivlist
  \item[\hskip\labelsep\itshape
    #1.]\ignorespaces
}{%
  \qed\endtrivlist
}
\newcommand{\proofname}{Proof}
\makeatother

\begin{document}

\title{A Characterization of Lyapunov Inequalities for Stability of Switched Systems} 
\author{Rapha\"{e}l M. Jungers, Amir Ali Ahmadi,\\ Pablo A. Parrilo and Mardavij Roozbehani}


\maketitle
\begin{abstract}      We study stability criteria for discrete-time \mrr{switched} systems \rmj{and provide a meta-theorem that characterizes all Lyapunov theorems \aaa{of a certain canonical type}. For this purpose,} we investigate the structure of sets of LMIs that provide a sufficient condition for stability.  \rmj{Various such conditions have been proposed in the literature in the \mrr{past} fifteen years. We prove in this note that a family of language-theoretic conditions recently provided by the authors encapsulates all the possible \mrr{LMI} conditions, thus putting a conclusion to this research effort.}\\  As a corollary, we show that it is PSPACE-complete to recognize whether a particular set of LMIs implies stability of a \mrr{switched} system.  Finally, we provide a geometric interpretation of these conditions, in terms of existence of an invariant set.
\end{abstract}

\section{Introduction} In this note, we study the structure and properties of stability conditions for \aaa{discrete time} \mrr{switched linear} systems. \aaa{These are systems that evolve according to the update rule:}
\begin{equation}\label{eq-switching}
x_{k+1}=A_{\sigma\left(k\right)  }x_{k},
\end{equation} 
\mrr{where the function $$\sigma(\cdot): \n\rightarrow \{1,\dots, m\} $$ is the \emph{switching signal} that determines which \aaa{square} matrix from the set $\setmat\mathrel{\mathop:}=\left\{ A_{1},...,A_{m}\right\} $ is applied to update the state at each time step}.  Thus, \mrr{the trajectory depends} on the particular values of the switching signal at times $k=1,2,\dots.$ \mrr{Switched} systems are a popular model for many different engineering applications.  
As a few examples, applications ranging from \aaa{viral disease treatment} optimization (\cite{hmcb10}) to \aaa{multi-hop} networks control (\cite{pappas-multihop}), or trackability of autonomous agents in sensor networks (\cite{cresp}) have been modeled with \aaa{switched linear} systems.  \rmjj{See also the survey \cite{garulli2012survey} for applications in e.g. video segmentation.}

\aaa{Unlike linear systems, many associated analysis and control problems for switched linear systems are known to be very hard to solve} (see \cite{BlTi3, jungers_lncis} and references therein). Among these, the \emph{Global Uniform Asymptotic Stability (GUAS) problem} \aaa{is a particularly fundamental and highly-studied question}. \aaa{We say that system \eqref{eq-switching}} is GUAS if all trajectories $x(t)$ tend to zero as $t\rightarrow \infty,$ \mrr{irrespective of the} switching law $\sigma(k)$. Because \mrr{switched} linear systems are \mrr{1-homogeneous}, there is no distinction between \mrr{local and global, or asymptotic and exponential} stability, \aaa{and for brevity we refer to the GUAS notion simply as \emph{stability} from here on}. \aaa{This stability property is fully} encapsulated in the so-called \emph{Joint Spectral Radius} (JSR) of the set $\setmat,$ which is defined
as
\begin{equation} \rho\left(\setmat\right)
=\lim_{k\rightarrow\infty}\max_{\sigma
\in\left\{  1,...,m\right\}  ^{k}}\left\Vert A_{\sigma_{k}}...A_{\sigma_{2}%
}A_{\sigma_{1}}\right\Vert ^{1/k}.\label{eq-def.jsr}%
\end{equation}

This quantity is independent of the norm
used in (\ref{eq-def.jsr}), and is smaller than \mrr{1} if and only if the system is stable. See \cite{jungers_lncis} for a recent survey on the topic. \mrr{In recent years much effort has been devoted to approximating this quantity}.  One of the most successful families of techniques for approximating the JSR \aaa{consists of} writing down a set of \aaa{inequalities}, whose parameters depend on the matrices defining the system, and which admit a solution \aaa{only if} the system is stable.  (That is, a solution to these \aaa{inequalities} is a \emph{certificate} for the stability of the system.) These \aaa{inequalities} are stated in terms of \aaa{linear programs, or semidefinite/sum of squares programs,} so that they can be solved with modern efficient convex optimization methods, like interior point methods (see for instance \cite{JohRan_PWQ,multiple_lyap_Branicky,AAA_MS_Thesis,Roozbehani2008,daafouzbernussou,bliman-trecate03,LeeD06,convex_conjugate_Lyap,Pablo_Jadbabaie_JSR_journal,protasov-jungers-blondel09}). 
Perhaps the simplest criterion in this family can be traced back to Ando and Shih \cite{ando-shih}, who proposed the following \aaa{inequalities}, as a stability certificate for a \mrr{switched system} described by a set of matrices\footnote{We note \aaa{by} $P\succ 0$ the constraint that $P$ is a symmetric, \aaa{positive definite} matrix. Also, throughout the note, we denote the transpose of the matrix $A$ by $A^{Tr},$ in order to avoid confusion with the power of a matrix (or a matrix set).} $\{A_i\}$ :
\begin{equation}\label{eq-Lyap.CQ.SDP}
\begin{array}{rll}
 A_i^{Tr}PA_i&\prec&P \quad i=1,\ldots,m.\\
P&\succ&0.
\end{array}
\end{equation}
It is \aaa{easily seen} that if these inequalities have a solution $P,$ then the function $x^{Tr}Px$ is a common quadratic Lyapunov function, meaning that this function decreases, \aaa{for all switching signals}.  This proves the following folklore theorem:
\begin{theom}
If a set of matrices $\setmat\mathrel{\mathop:}=\left\{ A_{1},...,A_{m}\right\} $ is such that the \aaa{inequalities in} (\ref{eq-Lyap.CQ.SDP}) have a \aaa{solution}, then the set is stable.
\end{theom} 

\aaa{The conditions in (\ref{eq-Lyap.CQ.SDP}) form a set of linear matrix inequalities (LMIs).} Other \aaa{types of} methods have been proposed to tackle the stability problem (e.g. variational methods \aaa{\cite{MM11}}, or iterative methods \aaa{\cite{GZalgorithm}}), but a great advantage of \aaa{LMI-based} methods is that (i) they offer a simple criterion that can be checked with the help of the powerful tools available for solving convex \aaa{(and in particular semidefinite)} programs, and \aaa{(ii) they} often come with a guaranteed accuracy \aaa{of approximation on the JSR} (see \cite{JSR_path.complete_journal}).
 
 \rmj{Recently, we proposed a whole family of \aaa{LMI-based, stability-proving} conditions and proved that they generalize all the previously proposed \aaa{criteria} that we were aware of (\cite{ajprhscc11,JSR_path.complete_journal}).}
\rmj{In this note, we \aaa{\emph{do not}} provide new criteria for stability.  Rather, we prove that the class of conditions recently proposed by us encapsulates all possible conditions \aaa{in a very broad and canonical family of LMI-based conditions (see Definition~\ref{def:Lyap.ineq}).} \rmjj{Note that in the conference version of this note \cite{ajprrocond}, we developed a proof for another (weaker) class of conditions, valid only for nonnegative matrices, which we call `entrywise-comparison Lyapunov functions.'}


In Section \ref{section-path-complete}, we recall the description of this class of conditions, namely the \emph{path-complete graph conditions}. In Section \ref{section-main}, we present our main result: no other \aaa{class of inequalities} than the ones \aaa{presented} in \cite{JSR_path.complete_journal} can be a valid stability criterion.  We then show that our result implies that recognizing if a set of LMIs is a valid criterion for stability is PSPACE-complete. In Section \ref{sec:invariant.sets}, \aaa{we further show how one can construct a single common Lyapunov function (or a geometric invariant set) for system (\ref{eq-switching}) given a feasible solution to the inequalities coming from \emph{any} path-complete graph.} 
\aaa{We end with a few concluding remarks in Section~\ref{section-conclusion}.}


\section{Stability conditions generated by path-complete graphs}\label{section-path-complete}}
Starting with the LMIs \aaa{in} (\ref{eq-Lyap.CQ.SDP}), many researchers have provided other criteria, based on semidefinite programming, for proving stability of \mrr{switched systems}.
The different methods amount to \mrr{writing down} \aaa{different sets} of \mrr{inequalities---formally defined below as \emph{Lyapunov \aaa{inequalities}}---}which, \aaa{if satisfied by a set of functions, imply that the set $\setmat$ is stable.}


 
\mrr{
\begin{defi}\label{def:Lyap.ineq}
Given a \mrr{switched system} of the \mrr{form} (\ref{eq-switching}), a \emph{Lyapunov inequality} is a quantified inequality of the \mrr{form}:
\rmjj{\begin{equation}\label{eq:LI} \forall x\in \re^n\aaa{\backslash\{0\}}, V_j(Ax)< V_i(x),  \end{equation}}
where the functions $V_i,V_j:\mathbb{R}^n\rightarrow\mathbb{R}$ are \emph{Lyapunov functions} (\aaa{always taken to be continuous, positive definite \rmjj{(i.e. $V(0)=0$ and $\forall x\neq 0, V(x)>0$)}, and homogeneous functions in this note}), and \aaa{the matrix} $A$ is a \aaa{finite product out of the} matrices in $\setmat.$ In the special case that the Lyapunov functions are quadratic \aaa{forms,} we refer to \eqref{eq:LI} as \aaa{a} quadratic Lyapunov \aaa{inequality}. 
\end{defi}
}

\aaa{We are interested in the problem of characterizing which finite sets of Lyapunov inequalities of the type (\ref{eq:LI}) imposed among a (finite) set of Lyapunov functions $V_1,\ldots, V_k$ imply stability of system (\ref{eq-switching}). The simplest example of such a set is the quadratic Lyapunov inequalities presented (in LMI notation) in (\ref{eq-Lyap.CQ.SDP}). Here, there is only a single Lyapunov function, and the matrix products out of the set $\Sigma$ are of length one.}



\mrr{
\begin{remark}\label{remark-approx} \aaa{Due to} the homogeneity of the definition of the JSR \aaa{under matrix scalings (see (\ref{eq-def.jsr}))}, one can derive an upper bound $\gamma^*$ on the joint spectral radius by applying the Lyapunov \aaa{inequalities} to the \emph{scaled} set of matrices \begin{equation}\label{eq-scaled-set} \setmat/\gamma =\{ A/\gamma: A\in \setmat\}, \end{equation} and taking $\gamma^*$ to be the minimum $\gamma$ such that (\ref{eq-scaled-set}) satisfies the inequalities \aaa{for some Lyapunov functions within a certain class (e.g., quadratics, quartics, etc.)}. The maximal real number $r$ satisfying \rmjj{$r\gamma^*\leq \rho(\setmat) \leq \gamma^*, $} for any arbitrary set of matrices then provides a \aaa{worst-case guarantee on the quality of approximation of a} particular set of Lyapunov inequalities.
In particular, it is
known \aaa{\cite{ando-shih}} that the estimate $\gamma^*$ obtained
with \eqref{eq-Lyap.CQ.SDP} satisfies
\begin{equation}\label{eq-CQ.bound}
\frac{1}{\sqrt{n}}\gamma^* \leq\rho(\setmat)\leq\gamma^* ,
\end{equation}
where $n$ is the dimension of the matrices.  
\end{remark}
}

In \aaa{a} recent paper \cite{JSR_path.complete_journal}, \rmj{we} have presented a framework in which all these methods find a common generalization.
Roughly speaking, the idea behind this general framework is that a set of Lyapunov inequalities \rmjj{actually characterizes} a set of switching signals for which the trajectory remains stable.  Thus, a \aaa{stability-proving} set of Lyapunov inequalities must encapsulate all the possible switching signals and provide \mrr{an abstract stability proof for all of these trajectories}.  One contribution of \cite{JSR_path.complete_journal} is to provide a way to represent such a set of Lyapunov inequalities with a directed labeled graph which represents all the switching signals that, \mrr{by virtue of the Lyapunov inequalities, are  guaranteed to make the system converge to zero.}  \mrr{Thus, in order to determine whether the corresponding set of inequalities is a sufficient condition for stability, \rmjj{one only has to check} that all the possible switching signals are represented in the graph.}


In the following, by a slight abuse of notation, the same symbol $\setmat$ can represent a set of matrices, or an alphabet of abstract characters, corresponding to each of the matrices. Also, for any alphabet $\setmat,$ we note $\setmat^*$ (resp. $\setmat^t$) the set of all words on this alphabet (resp. the set of words of length $t$). Finally, for a word $w\in \setmat^t,$ we \aaa{let $A_w$ denote} the product corresponding to $w:$ $A_{w_1}\dots A_{w_t}$.

We represent a set of Lyapunov inequalities on a directed
labeled graph $G.$ Each node of \aaa{$G$} corresponds to a \aaa{single} Lyapunov function
$V_i$ and each \aaa{of its edges, which is
labeled by a finite product of matrices from $\Sigma$, i.e., by a word from the
set $\setmat^*$, represents a single Lyapunov inequality.} As illustrated in Figure~\ref{fig-node.arc}, for any word $w\in \setmat^*,$ and any Lyapunov inequality of the \mrr{form}
\begin{equation}\label{eq-lyap.inequality.rule}
 \forall x\in\mathbb{R}^n\aaa{\backslash\{0\}},\quad V_j(A_wx)< V_i(x),
\end{equation}
we add an arc going from node $i$ to node $j$ labeled with the word
$\bar w$ (the \emph{mirror} $\bar w$ of a word $w $ is the word obtained by reading $w$ \mrr{backwards}).  So, for a particular set of Lyapunov inequalities, there are as many nodes in the graph as there are (unknown) \aaa{Lyapunov} functions $V_i,$ and as many arcs as there are inequalities.

\begin{figure}[ht]
\centering \scalebox{.1} {\includegraphics{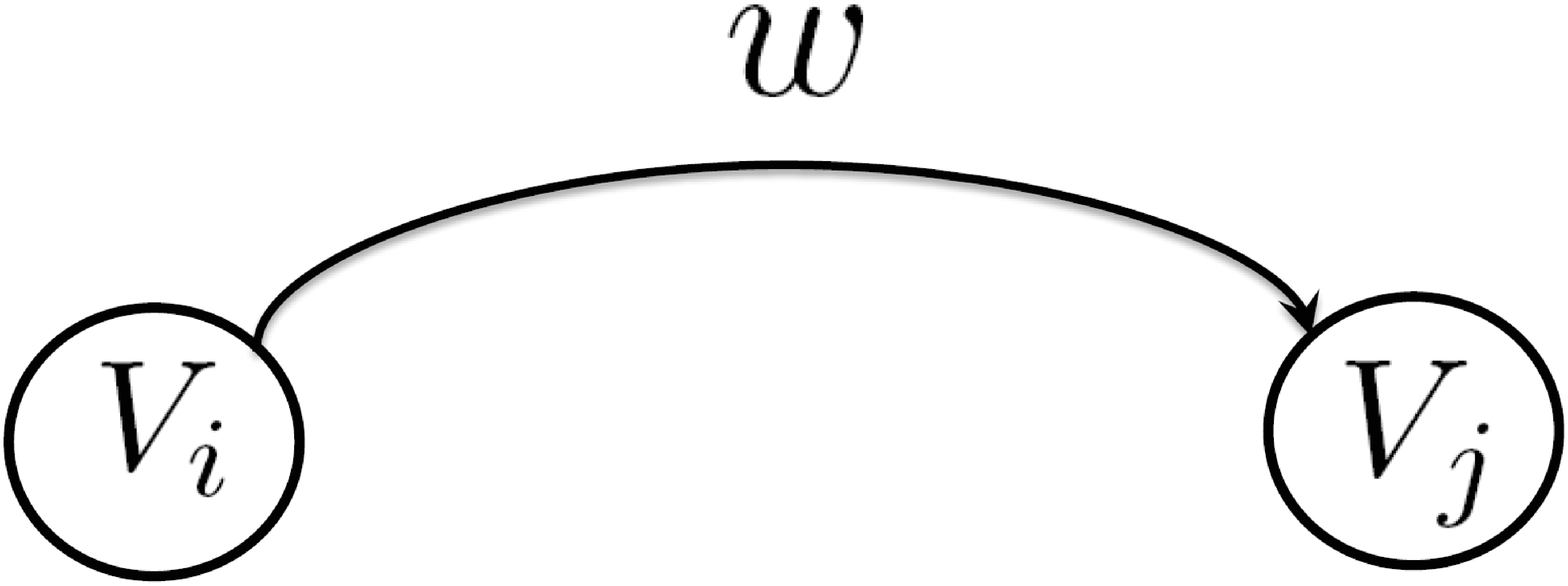}}
\caption{Graphical representation of a \aaa{single Lyapunov inequality}. \rmjj{The
graph above corresponds to the Lyapunov inequality $V_j(A_{\bar w}x)<
V_i(x)$. Here, $A_{\bar w}$ can be a single matrix from $\setmat$ or
a finite product of matrices from $\setmat$.}}
\label{fig-node.arc}
\end{figure}
\aaa{The reason for this construction is that we will reformulate the question of testing whether a set of Lyapunov inequalities provides a sufficient condition for stability, as that of checking a certain property of this graph. This brings us to the notion of \emph{path-completeness}, as defined in~\cite{JSR_path.complete_journal}.}

\begin{defi}\label{def-path-complete}
Given a directed graph $G$ whose arcs are labeled with words
from the set $\setmat^*$, we say that the graph is
\emph{path-complete} if for any finite word $w_1\dots w_k$ of any length $k$ (i.e., for all words in
$\setmat^*$), there is a directed path in $G$ such that the word obtained by concatenating the labels of the edges on this path contains the word $w_1\dots w_k$ as a subword.
\end{defi}
\aaa{The connection to stability is established in the following theorem:}

\aaa{
\begin{theom}\label{thm-path.complete.implies.stability}(\cite{JSR_path.complete_journal})
Consider a set of matrices
$\Sigma=\{A_1,\ldots,A_m\}$. Let $G$ be a path-complete
graph whose edges are labeled with words from
$\Sigma^*$. If there exist Lyapunov functions
$V_i$, one per node of the graph, that satisfy the Lyapunov inequalities represented by each edge of the graph, then the switched system in (\ref{eq-switching}) is stable.
\end{theom}
}


\begin{example}
The graph \aaa{depicted} in Figure \ref{fig-hscc} \rmjj{(a graph with $m=2$ labels)} is path-complete: one can check that every word can be read \aaa{as a path} on this graph.  As a consequence, the \aaa{following} set of \aaa{LMIs} is a valid sufficient condition for stability:
\end{example}
\begin{figure}[ht]
\centering \scalebox{.3} {\includegraphics{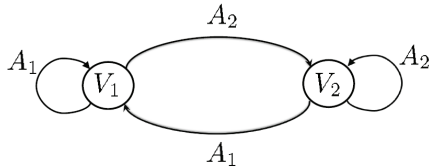}}
\caption{A graph corresponding to the LMIs \aaa{in (\ref{eq-hscc})}.  The graph is path-complete, and as a consequence any \mrr{switched system} for which these LMIs have a solution is stable.}
\label{fig-hscc}
\end{figure}

\begin{equation}\label{eq-hscc}
\begin{array}{rll}
A_1^{Tr}P_1A_1&\prec&P_1 \\
A_1^{Tr}P_1A_1&\prec&P_2 \\
A_2^{Tr}P_2A_2&\prec&P_1 \\
A_2^{Tr}P_2A_2&\prec&P_2 \\
P_1,P_2&\succ&0.
\end{array}
\end{equation}

\aaa{Unlike this simple example, the structure of path-complete graphs can be quite complicated, leading to rather nontrivial sets of LMIs that imply stability; see the examples in~\cite{JSR_path.complete_journal} such as Proposition 3.6.}

In this note, we \aaa{mainly} investigate the converse \aaa{of} Theorem \ref{thm-path.complete.implies.stability} and answer the question ``Are there other sets of \aaa{Lyapunov inequalities} which do not correspond to path-complete graphs but are sufficient conditions for stability?'' \mrr{The answer is negative.} Thus, \aaa{path-completeness fully characterizes the set of all stability-proving Lyapunov inequalities}.
\aaa{By} Remark \ref{remark-approx}, \mrr{this also leads to a complete characterization of all valid inequalities for \aaa{approximation} of the joint spectral radius. } 

To this purpose, \mrr{in the next \aaa{section} we} show that for any non-path-complete graph, there exists a set of matrices which is not stable, but yet \aaa{makes the corresponding Lyapunov inequalities feasible.} This is not \mrr{a trivial} task a priori, because we need to \aaa{construct} a counterexample without knowing the graph \mrr{explicitly}, but just with the information that it is not path-complete.  \aaa{With this information alone we need to provide two things: (i) a set of matrices that are unstable, and (ii) a set of Lyapunov functions $V_i$ such that the Lyapunov inequalities associated to the edges (with the matrices found in (i)) are feasible.}




\section{The main result \aaa{(necessity of path-completeness)}}\label{section-main}

\subsection{The construction}\label{section-construction}

\aaa{As stated just above,} we want to prove that if a graph is not path-complete, it \aaa{does not provide} a valid criterion for stability, meaning that there must exist \mrr{an unstable} set of matrices that satisfies the corresponding Lyapunov inequalities.  Our goal in this subsection is to describe a simple construction that will allow us to build such a \mrr{set.}
If a graph is not path-complete, there is a certain word $w$ which cannot be \aaa{``read''} on the graph.
We propose a simple construction of a set of matrices with the following property: any long product of these matrices which is not equal to the zero matrix must contain the product $A_w.$  
\begin{defi}\label{defi-sigmaomega}
\mrr {Let $w\in \{1,2,\dots, r\}^*$ be a word on an alphabet of $r$ characters and let $n=|w|+1$.} We \mrr{denote by} $\setmat_w$ the set \mrr {of $n\times n$} $\{0,1\}$-matrices\footnote{By `$\{0,1\}$-matrix' we mean a matrix \rmjj{with all entries equal to zero or one}, and $|w|$ denotes the length of the word $w.$} $\{A_1,A_2,\dots, A_r\}$ such that the $(i,j)$ entry of $A_{l}$ is equal to one if and only if
\begin{itemize}
\item \rmjj{$j=i+1, \mbox{ and } w_i=l,$ for $1\leq i\leq n-1,$ }

 { or }
\item $(i,j)=(n,1)$ and $l=1.$
\end{itemize}\end{defi}

\begin{figure}
\centering \scalebox{0.1} {\includegraphics{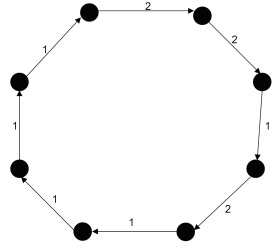}}
\caption{Graphical representation of the construction of the set of matrices $\setmat_w$ for $w={2212111}:$ the edges with label $1$ represent the matrix $A_1$ (i.e. $A_1$ is the adjacency matrix of the subgraph with edges labeled with a ``$1$''), and the edges with label $2$ represent the matrix $A_2.$}
\label{fig-sigmaomega}
\end{figure}
\mrr{In other words}, $\setmat_w$ is the only set of binary matrices whose sum is the adjacency matrix of the cycle on $n$ nodes, and such that for all $i\in\{1,\dots,n-1\},$ the $i$th edge of this cycle is in the graph corresponding to $A_{w_i},$ the last edge being in the graph corresponding to $A_1.$  Figure \ref{fig-sigmaomega} provides a visual representation of the set $\setmat_w.$

\mrr{The following lemma characterizes the main property of our construction $A_w$ in a straightforward fashion.}

\begin{lem}\label{lem-subproduct}
Any nonzero product in $\setmat_w^{2n}$ contains $A_w$ as a subproduct.
\end{lem}
\begin{proof}
\rmjj{Recall that the matrices in $\setmat_w$ are adjacency matrices of a subgraph of the cycle on $n$ nodes.
Hence, a nonzero product corresponds to a path in this graph.  A path of length more than $n$ must contain a cycle. Since there is only one cycle in this graph, this cycle is the whole graph itself.}  Finally, a path of length $2n$ must contain a cycle starting at node $1$ (i.e., the node corresponding to the beginning of the word $w$).  Hence, this product contains $A_w.$\end{proof}

\subsection{The proof}

Let us consider an arbitrary non-path-complete graph.  From this graph, we will first construct a set of matrices which is not stable.  Then we will prove that the corresponding \aaa{inequalities}, as given by our automatic recipe that \aaa{associates graphs with Lyapunov inequalities via Figure \ref{fig-node.arc}}, admit a solution. \aaa{For this purpose, we will take our Lyapunov functions to be quadratic functions (though a similar construction is possible for other classes of Lyapunov functions).}
Thus, we will have to find a solution (that is, a set of \aaa{positive definite} matrices $P_i$) to particular sets of LMIs.  It turns out that for our purposes, we can restrict our attention to \rmjj{diagonal matrices} $P=\mbox{diag}(p),$ where \aaa{$p\in \re^n$} is a positive vector. For these matrices, the corresponding Lyapunov functions satisfy\footnote{\rmjj{To avoid conflict with other notation, we write $p(l)$ for the $l$th entry of vector $p.$}}:\begin{equation}\label{eq-diag-quads}V_{p}(x)=\sum_{i=1}^np(i)x_{i}^2. \end{equation}

The following proposition provides an easy way to express Lyapunov inequalities \aaa{among diagonal quadratic forms}. It will allow us to write the Lyapunov inequalities in terms of entrywise vector inequalities.  

\begin{prop}\label{lem-linear-norms}
\rmjj{Let $p,p'\in \re^n_{++}$ (i.e. $p,p'$ are positive vectors), and $A\in \{0,1\}^{n\times n}$ be matrices with not more than one nonzero entry in every row and every column.  Then, we have
\begin{equation}\label{eq-conicnorm-lp} \forall x\in \re^n, \, V_{p'}(A^{Tr}x)< V_{p}(x)\quad \iff \quad Ap' < p ,\end{equation} where the vector inequalities are to be understood componentwise, and the Lyapunov functions $V_p$ \aaa{and $V_{p'}$} are as in Equation \eqref{eq-diag-quads}.} 
\end{prop}
\begin{proof}
$\Rightarrow:$ For an arbitrary index $1\leq l'\leq n,$ consider the $l'th$ row in the right-hand side inequality.  If the $l'th$ row of $A$ is equal to the zero vector, then the inequality is obvious.  If not, take the index $1\leq l\leq n$ such that $A_{l',l}=1,$ and fix $x=e_{l'},$ \rmjj{that is, the $l'$th canonical basis vector.}  Then the left-hand side of (\ref{eq-conicnorm-lp}) becomes \rmjj{$$p'(l)< p(l'), $$} which is exactly the $l'th$ row of the entrywise inequality we wish to prove.

$\Leftarrow:$
 Writing $P'=\mbox{diag}(p'),$ we have: \rmjj{ \begin{eqnarray}AP'A^{Tr}&=& A \left (\sum_l{p'(l)\vectun_l\vectun_l^{Tr}}\right ) A^{Tr}\\ \nonumber &= &   \sum_l{p'(l) (A\vectun_l)( A \vectun_l)^{Tr}}\\ \nonumber&\prec&  \sum_{l'}{p(l')} \vectun_{l'}\vectun_{l'}^{Tr}\\ \nonumber&=&P, \end{eqnarray}
which is equivalent to the left-hand side of (\ref{eq-conicnorm-lp}).}
\end{proof}

In the previous subsection, we have shown how to build, for a particular non path-complete graph, a set of matrices, which is clearly not stable.  In the above proposition, we have shown that, for such matrices (with not more than one one-entry in every row and every column), the Lyapunov inequalities translate into simple constraints (namely, entrywise vector inequalities), provided that the quadratic Lyapunov functions $P_i$ have the simple \mrr{form} of Equation \eqref{eq-diag-quads}.  Now, we put all these pieces together: we show how to construct such Lyapunov functions, and use our characterization \eqref{eq-conicnorm-lp} in order to prove that these functions satisfy the Lyapunov inequalities.
\begin{theom}\label{theo-stab-implies-pc} A set of quadratic Lyapunov inequalities is a sufficient condition for stability if and only if the corresponding graph is path-complete.
\end{theom}
\begin{proof}
The if part is exactly Theorem \ref{thm-path.complete.implies.stability}. We now prove the converse: for any non path-complete graph, we constructively provide a set of matrices that satisfies the corresponding Lyapunov inequalities (with Lyapunov functions of the form \eqref{eq-diag-quads}), but which is not stable.  \rmjj{Our proof works in three steps: first, for a given graph which is not path-complete, we show how to build a particular unstable set of matrices. Then, we compute a set of solutions $p_i$ for our Lyapunov inequalities, and finally, we prove that these $p_i$ are indeed valid solutions, for the particular matrices we have built.}

{\bf 1. The counterexample}\\
For a given graph $G$ which is not path-complete, there is a word $w$ that cannot be read as a subword of a sequence of labels on a path in
this graph.  We use the construction above with the particular word $ w.$  We show below that the set of Lyapunov inequalities corresponding to $G$ admits
a solution for the set of matrices (see Definition \ref{defi-sigmaomega}) $$\setmat_{  w}^{Tr}=\{A^{Tr}:A\in \setmat_w\}.$$ Actually, we show that for these matrices, there is in fact a solution within the restricted family of diagonal quadratic Lyapunov functions defined in \eqref{eq-diag-quads}.  Since $\setmat_{  w}^{Tr}$ is not stable by construction, this will conclude the proof.

{\bf 2. Explicit solution of the Lyapunov inequalities}\\
We have to construct a vector $p_i$ defining a norm for each node of the graph $G.$ In order to do this, we construct an \emph{auxiliary graph} $G'(V',E')$ from the graph $G.$  The nodes of $G'$ are the couples ($N$ is the number of nodes in $G$ and $n$ is the dimension of the matrices in $\setmat_{  w}^{Tr}$): $$V'=\{(i,l): 1\leq i\leq N, 1\leq l\leq n\} $$ (that is, each node represents a particular entry of a particular Lyapunov function $p_i$).
There is an edge in $E'$ from $(i,l)$ to $(j,l')$ if and only if \begin{enumerate}  \item \label{item-label}there is an edge from $i$ to $j$ in $G$ with label $A_k$ (\rmjj{where $A_k$ can represent a single matrix, or a product of matrices}),\item the corresponding matrix $A_k$ ($A_k$ is a matrix in $ \setmat_w,$ or possibly a product of such matrices) is such that
\begin{equation}\label{eq-auxiliary-graph} (A_k)_{l,l'}=1.\end{equation} \end{enumerate} We give the label $A_k$ to this edge in $G'.$

We \emph{claim} that $G'$ is acyclic. Indeed, by (\ref{eq-auxiliary-graph}), a cycle $(i,l)\rightarrow \dots \rightarrow (i,l)$ in $G'$ describes a product of matrices in $\setmat_{  w}$ such that $A_{l,l}=1.$  (We can build this product by following the labels of the cycle.) Now, take a nonzero product of length $2n$ by following this cycle (several times, if needed).  By Lemma \ref{lem-subproduct}, any long enough nonzero product of matrices in $\setmat_{  w}$ contains the product $A_{  w}; $ and thus there is a path with label $w$ in $G'.$\\
Now, by item \ref{item-label}. in our construction of $G',$ any such path in $G'$ corresponds to a path in $G$ with the same sequence of labels $w,$ a contradiction; and \emph{this proves the claim.} 

Let us construct $G'=(V',E')$ as above.  It is well known that the nodes of an acyclic graph admit a renumbering $$s: \, V\rightarrow \{1,\dots,|V|\}:\quad v\rightarrow s(v) $$ such that there can be a path from $v$ to $v'$ only if $s(v)>s(v')$ (see \cite{kahn62}).  This numbering finally allows us to define our nonnegative vectors $p_i$ in the following way:
$$ p_i({l}) := s((i,l)).$$

{\bf 3. Proof that the solution $\{p_i\}$ is valid}\\
We have to show that for every edge $i\rightarrow j$ of $G=(V,E)$ with label $A_k,$ the following holds  $$\forall x, V_{p_j}(A_k^{Tr}x)< V_{p_i}(x).$$
By Proposition \ref{lem-linear-norms} above, we can instead show that \begin{equation}\label{eq-entrywise} A_kp_j< p_i. \end{equation} 
 
If $(A_kp_j)_{l}= 0,$ then \eqref{eq-entrywise} obviously holds at its $l$th component.  If $(A_kp_j)_{l}\neq 0,$ we have a particular index $l'$ such that $(A_k)_{l,l'}=1,$ and $$(A_kp_j)_{l}=(p_j)_{l'}.$$  Now, it turns out that indeed $$(p_j)_{l'}< (p_i)_{l.}$$ This is because $(A_k)_{l,l'}=1,$ together with $(i,j)\in E$ implies (by our construction) that there is an edge $((i,l)\rightarrow (j,l'))\in E'.$ Thus, $(p_j)_{l'}< (p_i)_{l},$ which gives the required Inequality \eqref{eq-entrywise}, and the proof is complete.
\end{proof}

\begin{example}
The graph represented in Figure \ref{fig-hscc-wrong} is \emph{not path-complete}: \rmjj{one can easily check for instance that the word $A_1A_2A_1$} cannot be read as a subword of a path in the graph.  As a consequence, the set of \aaa{LMIs in} (\ref{eq-hscc-wrong}) is not a valid condition for stability, even though it is very much similar to (\ref{eq-hscc}).

\begin{equation}\label{eq-hscc-wrong}
\begin{array}{rll}
A_1^{Tr}P_1A_1&\prec&P_1 \\
A_2^{Tr}P_1A_2&\prec&P_2 \\
A_2^{Tr}P_2A_2&\prec&P_1 \\
A_2^{Tr}P_2A_2&\prec&P_2 \\
P_1,P_2&\succ&0.
\end{array}
\end{equation}
As an example, one can check that the set of matrices  $$\setmat =\left \{  \begin{pmatrix}
    -0.7   &  0.3  &   0.4\\
     0.4   &  0 &    0.8\\
    -0.7   &  0.5 &    0.7
\end{pmatrix},  
\begin{pmatrix} -0.3   & -0.95   &      0\\
    0.4   & 0.5 &   0.8\\
   -0.6       &  0   & 0.2
\end{pmatrix}\right \} $$  makes (\ref{eq-hscc-wrong}) feasible, even though this set is unstable.  Indeed, $$\rho(\Sigma)\geq \rho(A_1A_2A_1)^{1/3}=1.01\dots$$\\ \rmj{
It is not surprising that this unstable product precisely corresponds to a ``missing'' word in the language generated by the automaton in Figure \ref{fig-hscc-wrong}.}
\end{example}
\begin{figure}[ht]
\centering \scalebox{.3} {\includegraphics{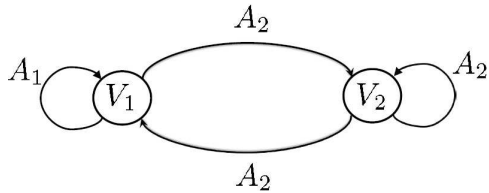}}
\caption{\aaa{The} \rmjj{graph corresponding to the LMIs \aaa{in} (\ref{eq-hscc-wrong}).  The graph is not path-complete: one can easily check for instance that the word $A_1A_2A_1$} cannot be read as a path in the graph.}
\label{fig-hscc-wrong}
\end{figure}

\subsection{PSPACE-completeness of the recognizability problem}
Our results imply that it is PSPACE-complete to recognize sets of LMIs that are valid stability criteria. \rmjj{PSPACE-complete problems are a well-known class of problems, harder than NP-complete, see \cite{GJ-computers-igt} for details.}
\begin{theom}\label{thm-pspace}
Given a set of quadratic Lyapunov inequalities, it is \aaa{PSPACE-complete} to decide whether they constitute a valid stability criterion.
\end{theom}
\begin{proof}  Our proof works by reduction from the \fulllanguage.  
In the \fulllanguage, one is given a finite state automaton on a certain alphabet $\Sigma,$ and it is asked whether the language that it accepts is the language $\Sigma^*$ of all the possible words. It is well known that the \fulllanguage{} is PSPACE-complete (\cite{GJ-computers-igt}).\\ A labeled graph corresponds in a straightforward way to a finite state automaton.  However, the concept of automaton is slighlty more general, in that they include in their definition a set of \emph{starting states,} and \emph{terminating states,} so that each path certifying that a particular word belongs to the language must start (resp. end) in a starting (resp. terminating) state. Thus, in order to reduce the \fulllanguage{}  to the question of recognizing whether a graph is path-complete, we must be able to transform the automaton into a new one for which all the states are starting and accepting. \rmjj{For this purpose, we will need to introduce a new fake character $f$ in the alphabet.}

So, let us be given an arbitrary automaton.  We then connect all accepting nodes to all starting nodes, with an edge labeled with \rmjj{the new character} $f.$  Now, we make all the nodes starting and accepting, and we ask whether all the words in our new alphabet $\Sigma \cup \{f\}$ are accepted in our new automaton, that is, if the obtained graph is path-complete. \\
If all words can be read on the graph, then in particular, \aaa{all words starting and ending with the character $f,$ and containing non-$f$ characters in between can be read on the graph}. Thus, our initial automaton generates all the words on the initial alphabet. \\
Conversely, if the initial automaton generates all the words on the initial alphabet, by decomposing an arbitrary word on the new alphabet \aaa{as} $w=w_1fw_2f\dots f,$ we see that we can generate all the words on the new alphabet $\Sigma \cup \{f\}.$  
In other words, the new graph is path-complete if and only if the initial automaton was accepting all the words on $\Sigma^*.$ The proof is complete.
\end{proof}

\section{Invariant sets and \aaa{path-complete graphs}}\label{sec:invariant.sets}

Recall that Theorem~\ref{thm-path.complete.implies.stability} states that Lyapunov inequalities associated with any path-complete graph imply stability of the switched system in (\ref{eq-switching}). The proof of this theorem, as it appears in~\cite{ajprhscc11},~\cite{JSR_path.complete_journal}, does not give rise to an explicit common Lyapunov function for system (\ref{eq-switching}). In other words, if the conditions of Theorem~\ref{thm-path.complete.implies.stability} are satisfied, then we know that system (\ref{eq-switching}) is \aaa{stable}, but it is not clear how to construct an invariant set for its trajectories. The question hence naturally arises (and has been repeatedly brought up to us in presentations of our previous work) as to whether one can construct a single common Lyapunov function $W$ (i.e., one that satisfies $W(A_ix)< W(x), \forall x\neq0, \forall i\in\{1,\ldots,m\}$) by combining the Lyapunov functions $V_i$ assigned to each node of the graph. This is the question that we address in this section.\aaa{\footnote{For the purposes of this section, we assume that the labels on our edges are matrix products of length one, though the extension to the general case is straightforward.}} We start with a simple proposition that establishes a ``bi-invariance'' property, showing that a path-complete graph Lyapunov function implies existence of two sets in the state space with the property that points in one never leave the other.


\begin{prop}\label{prop:bi.invariance}
Consider any path-complete graph $G$ with edges labeled by\\ $\{A_1,\ldots,A_m\}$ and suppose there exist Lyapunov functions $V_1,\ldots,V_k$, one per node of $G$, that satisfy the inequalities imposed by the edges. Then, for all $\alpha\geq 0,$ and for all finite products $A_{\sigma_s}\cdots A_{\sigma_1}$, if 
$\bar{x}\in\cap_{l=1,\ldots,k} \{x| \ V_l(x)\leq\alpha\},$ then
$$A_{\sigma_s}\cdots A_{\sigma_1} \bar{x}\in\cup_{l=1,\ldots,k} \{x| \ V_l(x)\leq\alpha\}.$$
\end{prop}

\begin{proof} \rmjj{This is an obvious consequence of the definition of Lyapunov inequalities.}
\end{proof}

Knowing that points in the intersection of the level sets of $V_i$ never leave their union \aaa{may be good enough for some applications}. Nevertheless, one may be interested in having a true invariant set. This is achieved in the following construction, which is completely explicit but comes at the price of increasing the complexity of the set.

\begin{theom}\label{thm:constructing.invariant.set}
For a finite set of matrices $\mathcal{A}=\{A_1,\ldots,A_m\}$ and a scalar $\gamma$, let $\mathcal{A}_\gamma\mathrel{\mathop:}=\{ \gamma
A_{1},\ldots,\gamma A_{m}\}$. Consider any path-complete graph $G$ with edges labeled by the matrices in $\mathcal{A}_\gamma$. Suppose there exist Lyapunov functions $V_1,\ldots,V_k$, one per node of $G$, that satisfy the inequalities imposed by the edges of $G$ \rmj{with $\gamma>1.$} Then one can explicitly write down (from the Lyapunov functions $V_1,\ldots,V_k$ and matrices $A_1,\ldots,A_m$) a single Lyapunov function $W(x)$, which is a sum of pointwise maximum of quadratic functions, and satisfies 
\begin{equation}\label{eq:Wdot.negative}
W(A_ix)< W(x), \forall i\in\{1,\ldots,m\}, \forall x\neq 0.
\end{equation}

%
%
\end{theom}

\begin{proof}
Since the Lyapunov functions $V_i$ are positive definite and continuous, there exist positive scalars $\alpha_i, \beta_i$, such that  $$\alpha_i\leq V_i(x)\leq \beta_i,$$ for all $x$ with $||x||=1$ and for $i=1,\ldots,k$. The scalars $\alpha_i,\beta_i$ can be explicitly computed by minimizing and maximizing $V_i$ over the unit sphere. (In fact, one can show that a finite upper bound on $\beta_i$ and a finite lower bound on $\alpha_i$ can be found by solving two polynomially-sized sum of squares programs; such bounds would be sufficient for our purposes.) Let
\begin{equation}\label{eq:xi}
\xi\mathrel{\mathop:}=\max_{i,j\in\{1,\ldots,k\}} \frac{\beta_i}{\alpha_j}.
\end{equation}
In \cite{JSR_path.complete_journal} (Equation 2.4 and below), the following upper bound is proven on the spectral norm of an arbitrary product of length $s$:

\begin{equation}\label{eq:bound.on.spectral.norm}
||A_{\sigma_s}\cdots A_{\sigma_1} ||\leq \xi^{\frac{1}{d_m}}\frac{1}{\gamma^s},
\end{equation}
where $d_m$ is the maximum degree of homogeneity of the functions $V_i$. Let $r$ be the smallest integer $s$ that makes the right hand side of the previous inequality less than one. Note that $r$ can be explicitly computed from $\gamma$ and $V_1,\ldots,V_k$, once the Lyapunov inequalities are solved. Let $E(x)\mathrel{\mathop:}=||x||^2.$ Then we claim that 
\begin{eqnarray}\nonumber
W(x)=E(x)+\max_{{i}\in\{1,\ldots,m\}}E(A_ix)+\max_{{i,j}\in\{1,\ldots,m\}^2}E(A_iA_jx)\\+\cdots+\max_{\sigma\in\{1,\ldots,m\}^{r-1}}E(A_{\sigma_{r-1}}\cdots A_{\sigma_1}x)
\end{eqnarray}
satisfies  (\ref{eq:Wdot.negative}). Indeed, for any $l\in\{1,\ldots,m\}$ and any $x\neq 0,$ 
\begin{equation}\nonumber
\begin{array}{ll}
W(A_lx)&= E(A_lx)+ \max_{{i}\in\{1,\ldots,m\}}E(A_iA_lx)\\ \ &+\cdots+\max_{\sigma\in\{1,\ldots,m\}^{r-1}}E(A_{\sigma_{r-1}}\cdots A_{\sigma_1}A_lx)\\ \ &\leq \max_{{i}\in\{1,\ldots,m\}}E(A_ix) +\cdots \\ \ &+\max_{\sigma\in\{1,\ldots,m\}^{r}}E(A_{\sigma_{r}}\cdots A_{\sigma_1}x)\\ 
&=W(x)-E(x)+\max_{\sigma\in\{1,\ldots,m\}^{r}}E(A_{\sigma_{r}}\cdots A_{\sigma_1}x)\\ 
&< W(x), 
\end{array}
\end{equation}
where the last inequality follows from (\ref{eq:bound.on.spectral.norm}) and the definition of $r$. 
\end{proof}

\rmjj{Note that we did not assume in the theorem above that the Lyapunov functions are quadratic, and indeed the proof works for general Lyapunov functions as considered in this paper.}

%

\section{Conclusion\aaa{s}}\label{section-conclusion}
There has been a \aaa{surge of research activity} in the last fifteen years to derive stability conditions for \mrr{switched systems}.  In this work, we took \aaa{this research direction} one step further. \aaa{We showed that \emph{all} stability-proving Lyapunov inequalities within a broad and canonical family can be understood by a single language-theoretic framework, namely that of path-complete graphs.} \aaa{Our work also showed that one can always use the Lyapunov functions associated with the nodes of a path-complete graph to construct an invariant set for the switched system.}
As explained \aaa{before}, our results are not only relevant for proving stability of \mrr{switched systems}, but also for the \aaa{goal} of approximating the joint spectral radius \aaa{of a set of matrices}.

A corollary of our main result is that \aaa{testing whether a set of equations provides a valid sufficient condition for stability is PSPACE-complete, and hence intractable in general}. In practice, however, one can choose to work with a fixed set of path-complete graphs whose path-completeness has been certified a priori. In view of this, we believe it is worthwhile to systematically compare the performance of different path-complete graphs, either with respect to all input matrices $\{A_i\}$, or those who may have a specific structure.  \rmjj{We leave that for further work.}

\rmjj{This note leaves open many questions, on which we are actively working: For example, our main result is devoted to LMI Lyapunov criteria.  Intuitively, one might expect that it would be valid for much more general Lyapunov functions.  However, our proof uses algebraic properties of these functions, so that it is not clear how to generalize it to more general families.}


\section{Acknowledgement}
We would like to thank Marie-Pierre B\'eal, Vincent Blondel, and Julien Cassaigne for helpful discussions leading to Theorem \ref{thm-pspace}.

\bibliographystyle{alpha}        

\bibliography{references}

\end{document}